\documentclass[10pt]{amsart}
\usepackage{amscd, amssymb}
\usepackage[all]{xy}
\usepackage{color}

\begin{document}
\newcommand{\mono}[1]{%
\gdef\puA{#1}}
\newcommand{\puA}{}
\newcommand{\faculty}[1]{%
\gdef\puC{#1}}
\newcommand{\puC}{}
\newcommand{\facultad}[1]{%
\gdef\puD{#1}}
\newcommand{\puD}{}
\newcommand{\N}{\mathbb{N}}
\newcommand{\Z}{\mathbb{Z}}
\newcommand{\gl}{\mbox{\rm{gldim}}}
\newtheorem{teo}{Theorem}[section]
\newtheorem{prop}[teo]{Proposition}
\newtheorem{lema}[teo] {Lemma}
\newtheorem{ej}[teo]{Example}
\newtheorem{obs}[teo]{Remark}
\newtheorem{defi}[teo]{Definition}
\newtheorem{coro}[teo]{Corollary}
\newtheorem{nota}[teo]{Notation}
\def\fin{\mbox{\rm{findim}}}
\newcommand{\fidim}{\mbox{\rm{$\phi$dim}}}
\newcommand{\psidim}{\mbox{\rm{$\psi$dim}}}
\def\mod{\mbox{\rm{mod}}}
\def\add{\mbox{\rm{add}}}
\def\maxi{\mbox{\rm{max}}}
\def\pd{\mbox{\rm{pd}}}
\def\gd{\mbox{\rm{gldim}}}
\def\mini{\mbox{\rm{min}}}
\newenvironment{note}{\noindent Notation: \rm}



\title[Igusa-Todorov for radical square]{Igusa-Todorov functions for radical square zero algebras\\
\it \small{Dedicated to Claude Cibils for his $60^{th}$ birthday}}

\author[Lanzilotta]{Marcelo Lanzilotta}
\address{Universidad de La Rep\'ublica,
Uruguay} \email{marclan@fing.edu.uy}

\author[Marcos]{Eduardo Marcos}
\address{Universidade de S\~ao Paulo\\
Brasil} \email{enmarcos@ime.usp.br}

\author[Mata]{Gustavo Mata}
\address{Universidad de La Rep\'ublica,
Uruguay} \email{gmata@fing.edu.uy}
\thanks{The third author was partially supported by a research grant of CNPq-Brazil, (bolsa de pesquisa),  also from a  grant from Fapesp-S\~ao Paulo, Brazil, and by MathAmSud, The first author had some visits to Brazil subsided by Proex IME-USP (CAPES), and by Fapesp}





\begin{abstract}{
\noindent In this paper we study the behaviour of the Igusa-Todorov functions for radical square zero algebras. We show that the left and the right $\phi$-dimensions coincide, in this case. Some general results are given, but we concentrate more in the radical square zero algebras. Our study is
based on two notions of  hearth and member of a quiver $Q$. We give some bounds for the $\phi$ and the $\psi$-dimensions and we describe the algebras for which the bound of $\psi$ is obtained. We also exhibit modules for which the $\phi$-dimension is realised.}
\end{abstract}
\maketitle

\section{Introduction}
One of the most important conjecture in the Representation Theory of Artin Algebras is the finitistic dimension, which states that the $\sup\{\pd(M): M$ is a finitely generated module of finite projective dimension$\}$  is finite. As an attempt to prove the conjecture Igusa and Todorov defined in  \cite{kn:igusatodorov}  two functions from the objects of $\mod A$ to the natural numbers, which generalises the notion of projective dimension, they are known, nowadays, as the Igusa-Todorov functions, $\phi$ and $\psi$.  One of its nicest features is that they are finite for each module, and they allow us to define the $\phi$-dimension and the $\psi$-dimension of an algebra. These are  new homological measures of the module category and in fact it holds that $\fin(A)\leq \phi\dim(A) \leq \psi \dim(A)\leq\gl(A)$ and they all coincide in the case of algebras with finite global dimension.\\
In \cite{kn:FLM} there are various relations of the $\phi$-dimension with the bifunctors $\mbox{\rm{ Ext}}( \cdot ,\cdot)$ and $\mbox{\rm{Tor}}(\cdot, \cdot)$, they also prove that the finiteness of the this dimension is invariant for derived equivalence.
 Recently various works were dedicated to study and generalise the properties of  these functions, see for instance
 \cite{kn:marchuard, kn:hlm1, kn:xu}, in particular in \cite{kn:xu} the Igusa-Todorov functions were defined for the derived category of an Artin Algebra.
 The calculation of the values  of these functions has not been done, up to know, for a large class of algebras. In this
 work we concentrate on the case of radical square algebras. Various notions related to this particular case are defined and worked out. We construct modules for which the $\phi$-dimension is realised. We also give bounds for
 $\phi$ and $\psi$ dimensions and discuss the cases where these bounds are obtained.
 We also give a complete characterisation of algebras with  maximal $\psi$-dimension.

\section{Preliminaries}

We start fixing some notation, $A$ will always denote an Artin algebra of  type $A = \frac{\mathbb{K} Q}{I}$ where $Q$ is a finite connected quiver, $I$ is an admissible ideal, and $\mathbb{K}$ a field, these algebras are called elementary, \cite{kn:ARS}. All modules will be finitely generated right modules. The category of finitely generated $A$ modules will be denoted by $\mod$A. We will denote by  $J$  the ideal generated by the arrows in $ \mathbb{K} Q$ and by $A_0 = \frac{\mathbb{K} Q}{J}$, which is  the sum of all simple modules, up isomorphism. Given an $A$-module $M$ we will denote its projective dimension by $\pd(M)$, and by $\Omega^n(M)$ and $\Omega^{-n }(M)$ the $n^{th}$ syzygy and $n^{th}$ cosyzygy of $M$ with respect to a minimal projective resolution and a minimal injective coresolution, respectively. We recall that the global dimension of $A$ which we will  denote by $\gd(A)$ is the supremum  of the set of projective dimensions of the $A$-modules,
 which is a natural number or infinite.  The finitistic dimension of $A$, denoted by $\fin(A)$  is the supremum of the set of projective dimensions of the $A$-modules with finite projective dimension.

\vspace{0.3cm}

\noindent We will divide the set of isomorphism classes of simple modules in three distinct sets which are the following:

\begin{itemize}
  \item $\mathcal{S} = \{S_1, \ldots, S_n\}$ denotes a complete set of simple $A$-modules, up to isomorphism.
  \item $\mathcal{S}_I$ the set of injective modules in  $\mathcal{S}$ .
  \item $\mathcal{S}_P$ the set of projective modules in $\mathcal{S}$.
  \item $\mathcal{S}_D = \mathcal{S}\setminus (\mathcal{S}_I \cup \mathcal{S}_P)$.
\end{itemize}

\begin{obs} \label{fgd}
Since we are considering only elementary algebras, in this work radical square zero algebra means algebras which are isomorphic to algebras of the type $A=\frac{\mathbb{K} Q}{J^2}$. For radical square zero algebras it holds that  $\displaystyle \Omega (S_i) = \bigoplus_{\alpha:i\rightarrow j} S_j$. That is $\Omega (S_i)$ is a direct sum  of simple modules and the number of summands isomorphic to $S_j$ coincides with the number of arrows starting in the vertex $i$ and ending at the vertex $j$. Using the last fact it is easy to compute that for a radical square zero algebra $A$ with
$n$ vertices and finite global dimension, its global dimension must be less or equal to $n-1$.
\end{obs}

\noindent The following subsection states some facts about rings given in triangular form, see \cite{kn:ARS} (Chapter III.2).

\subsection{Upper triangular matrices rings}

Given two rings $S$ and $T$ and a $S-T$-bimodule $M$, one constructs the  \textbf{upper triangular matrix ring} $A$, which is the set of matrices:

$$ \left\{ \left(
      \begin{array}{cc}
        s & m \\
        0 & t \\
      \end{array}
    \right): s\in S, \ m \in M \hbox{ and } t\in T
 \right\} $$

\noindent with the  usual addition and multiplication.

\begin{nota}
Given an algebra $A$ in the upper triangular form, as above,  we denote by  $\mathcal{P}(T)$ the set of isomorphism classes of indecomposable projective $T$ modules and by   $\{P_1 \ldots P_k\}$ the set of classes of indecomposable projective  $A$  modules which are not  $T$-modules.
\end{nota}

\begin{obs}\label{extensionconocida}(\cite{kn:ARS} Chapter III.2)

If $A$ is an upper triangular algebra, as above

\begin{itemize}
  \item  $\mod T \subset \mod A$ and $\mathcal{P}(A) = \mathcal{P}(T) \cup \{P_1 \ldots P_k\}$.
  \item Every projective $T$ resolution of a $T$-module $M$ is also a projective resolution of $M$ as $A$-module.

\end{itemize}

\end{obs}

\section{ Igusa-Todorov functions}

In this section,  we show some general facts about the Igusa-Todorov functions, for an Artin algebra $A$. Our objective here is to introduce some material which we will use in the following sections, where we will treat the case of algebras with radical square zero.

\begin{defi}
Let $K_0$ denote the quotient of the free abelian group generated by one symbol  $[M]$,  for each isomorphism class of right finitely generated $A$-module, and  relations given by:
\begin{enumerate}
  \item $[M]-[M']-[M'']$ if  $M \cong M' \oplus M''$.
  \item $[P]$ for each projective.
\end{enumerate}
\end{defi}

\noindent The following observation will be used frequently.

\begin{obs}\label{positivecone}
It is clear that $K_0$ has basis the set of symbols $[M]$, one for each class of isomorphism of indecomposable not projective module. Moreover every element in $K_0$ can be written in the form $[M]-[N]$, for some pair of modules $M$ and $N$.
\end{obs}

\noindent Let $\overline{\Omega}: K_0 \rightarrow K_0$ be the group endomorphism induced by $\Omega$, and let   $K_i = \overline{\Omega}(K_{i-1})= \ldots = \overline{\Omega}^{i}(K_{0})$. Now if $M$ is a finitely generated $A$ module then $\langle add M\rangle$ denotes the subgroup of $K_0$ generated by the classes of indecomposable summands of $M$.

\begin{defi}\label{monomorfismo}
The \textbf{(right)  Igusa-Todorov function } of $M\in \mod A$  is defined as  $\phi_{r}(M) = min\{l:
\overline{\Omega}{|}_{{\overline{\Omega}}^{l+s}\langle add M\rangle}$ is a monomorphism for all $s \in \mathbb{N}\}$.\\
\end{defi}

\noindent In an analogous way, using the cosygyzy, we may define the left Igusa-Todorov function $\phi_{l}(M)$. Using  duality we can see that
$\phi_{l}(M)=\phi_{r}(D(M))$, where $D$ is the usual duality between mod-$A$ and mod-$A^{op}$. In case there is no possible misinterpretation we will use the notation $\phi$ for the right Igusa-Todorov function.

\noindent The next lemma is a  generalisation of Lemma 0.2 of \cite{kn:igusatodorov}.

\begin{lema}\label{cuentita}
 If $M = P\oplus M_1^{l_1}\oplus \cdots\  \oplus M_t^{l_t}$ where $P$ is projective, $M_i$ is not projective for each $i$ and also  $M_i \ncong M_j$ for $i\neq j$ then $\phi(M) = \phi(M_1\oplus\cdots\oplus M_t)$. \end{lema}
\begin{proof}
The proof is analogous to the proof of the Lemma 0.2 \cite{kn:igusatodorov}.
\end{proof}

\noindent From now on when computing $\phi(M)$ we will assume that $M$ is basic and has no projective summand.

\begin{obs}\label{Omega}
Observe that if  $\Omega(M) \in add M$ then  $\overline{\Omega}(\langle addM\rangle) \subset \langle add M\rangle$.\\ We will consider next, various situations where we will
have a module $M$ such that $\overline{\Omega}(\langle addM\rangle) \subset \langle add M\rangle$. Observe that for radical square zero algebras this is the case for the semisimple module $A_0$.
\end{obs}

\begin{prop}\label{invariante}
Let  $M$ be an $A$-module such that  $\Omega(M) \in add M$, then the following statements are valid:
\begin{enumerate}
\item   $\phi(M) = min\{l\in \N: \overline{\Omega}{|}_{{\overline{\Omega}}^{l}\langle add
M\rangle }$ is injective$\}$.
\item $\phi (M) = min\{ l \in \N : rk (\overline{\Omega}^{l}{|}_{\langle add M\rangle} )= rk(\overline{\Omega}^{l+1}{|}_{\langle add M\rangle})\}$.
\item  If  $M = \oplus_{i=1}^k M_i$ is a decomposition of $M$ into indecomposable modules then  $\phi(M)\leq k$. Moreover if the equality holds then $ \pd(M) = k$.
\end{enumerate}
\end{prop}
\begin{proof}
The proof of the first statement is given by contradiction. Assume that there is an element
$x \in {\overline{\Omega}}^{l+m}(\langle add M\rangle)$,  and $m\in \N$ with
$\overline{\Omega}(x) = 0$. Using  Remark \ref{Omega}, we  see that  $\overline{\Omega}^{k}(\langle addM\rangle) \subset \langle add M\rangle$, for all $k \in \mathbb{N}$. From this we see that  $x \in {\overline{\Omega}}^{l}(\langle
add M\rangle)$ which implies that  $\overline{\Omega}{|}_{{\overline{\Omega}}^{l}\langle add M\rangle}$ is not injective and we get a contradiction.
Therefore, $\overline{\Omega}{|}_{{\overline{\Omega}}^{l+m}\langle add M\rangle}$ is injective for all  $m \in
\mathbb{N}$.\\
It is clear that $rk (\overline{\Omega}^{k}{|}_{\langle add M\rangle} )= rk(\overline{\Omega}^{k+1}{|}_{\langle add M\rangle})$  if and only if   $\overline{\Omega}{|}_{{\overline{\Omega}}^{k}\langle add M\rangle }$ is injective, and this shows the second statement.\\
\noindent Now let us show the last statement.
The rank of the abelian group $\langle add M\rangle$ is less or equal to $k$. Therefore if  $\overline{\Omega}|_{\langle add M\rangle}$ is not an isomorphism we get that $rk(\overline{\Omega}(\langle add M\rangle)) \leq k-1$.\\
\noindent We need the following claim:\\
\noindent \emph{ Claim: } $rk(\overline{\Omega}^s(\langle add M\rangle))= k-s$ for $0 \leq s \leq \phi(M)$.\\
Observe that  $rk(\langle addM\rangle) = k$, therefore  $rk(\overline{\Omega}^s(\langle add M\rangle))\leq k-s$, otherwise there would exist  $0 \leq s_0 < s$ such that  $\overline{\Omega}|_{\overline{\Omega}^{s_0}(\langle add M\rangle)}$ would be a monomorphism, then accordingly with item 1, $\phi (M) \leq s_0 < s\leq\phi(M)$, a contradiction.\\
\noindent The second part of the third statement is obtained by taking  $s=k$ in the previous claim.
\end{proof}

\begin{prop}Let $M$ be an $A$-module with $\Omega^{n}(M) \cong M$, for some $n\in \mathbb{N^*}$. Assume moreover that all indecomposable summand of $M$ have projective dimension bigger than $n$, then, $\phi(M) = 0$, in particular all direct summands of $M$ have infinite projective dimension.

\begin{proof}
Since $\Omega^{n}(M) \cong M$ we have that:

\begin{enumerate}
  \item If $\Omega^n(N)=0$, for some $N$ indecomposable summand of $M$ then its projective dimension is less or equal $n$, which contradicts our hypothesis.
  \item If  there is an indecomposable summand $N$ of $M$ such that $\Omega^n(N)$ is not indecomposable, then there is another indecomposable summand $N'$ of $M$ such that  $\Omega^n(N')=0$, which also contradicts our hypothesis.
\end{enumerate}

\noindent From the above we get that $\overline{\Omega}^{n}|_{\langle add M\rangle}: \langle add M\rangle \rightarrow \langle add M\rangle$ is an isomorphism, which implies the result.
\end{proof}
\end{prop}

\begin{defi} The \textbf{$\phi$-dimension} of an algebra $A$ is defined as follows \\ 

$$\fidim(A) = \sup\{\phi(M): M\in modA\}.$$
\end{defi}

\begin{prop}
If $A$ is an Artin algebra, then $\fidim(A) = \min \{l:\overline{\Omega}{|}_{K_l}$ is injective$\} $.
\end{prop}
\begin{proof}

\noindent If $\overline{\Omega}{|}_{K_i}: K_i \rightarrow K_{i+1}$ is not injective then  there are $N_1, N_2 \in \mod A$ such that  $[N_1]-[N_2] \in K_i$ is not zero and $\overline{\Omega}([N_1]-[N_2]) = 0$. Let  $M \in \mod A$ such that  $M_1, M_2 \in \add M$ and $\overline{\Omega}^{i}([M_1]-[M_2])= [N_1]-[N_2]$. It follows that  $\phi(M) \geq i+1$ since $\overline{\Omega} {|}_{\overline{\Omega}^i{(\langle \add M\rangle)}}$ is not an isomorphism, therefore $\phi \dim(A) \geq i+1$.  In particular one sees that $\phi \dim(A) \geq \maxi \{l:\overline{\Omega}{|}_{K_l}$ is not injective $\} + 1$.\\
\noindent If  $\overline{\Omega}{|}_{K_h}: K_h \rightarrow K_{h+1}$ is injective then $\overline{\Omega}{|}_{K_{h+j}}: K_{h+j} \rightarrow K_{h+j+1}$ is injective for all $j>0$.
From this we get that $\phi \dim (A) \leq h$ in particular $\phi \dim(A) \leq \mini \{l:\overline{\Omega}{|}_{K_l}$ is injective$\}$.\\
\noindent Finally since  $\maxi \{l:\overline{\Omega}{|}_{K_l}$ is not injective $\} + 1 = \mini \{l:\overline{\Omega}{|}_{K_l}$ is injective$\}$, the result follows.
\end{proof}

\begin{prop}\label{epimorfismo}
If $\overline{\Omega}^{k}|_{\langle add M\rangle}: \langle add M\rangle \rightarrow \langle add \Omega^{k}(M)\rangle$ is an epimorphism and  $rk(\langle \overline{\Omega}^{k} (add M) \rangle ) < rk(\langle \overline{\Omega}^{k-1}(add M) \rangle )$ then  $\phi(M) = \phi(\Omega^{k}(M))+k$.

\begin{proof}
Since $\overline{\Omega}^{k}|_{\langle add M\rangle}: \langle add M\rangle \rightarrow \langle add\Omega^{k}(M)\rangle$ is an epimorphism,  we get that $\overline{\Omega}^k(add M)  = add (\Omega^{k} (M)) $. Moreover since  $rk(\langle \overline{\Omega}^{k} (add M) \rangle ) < rk(\langle \overline{\Omega}^{k-1}(add M) \rangle )$ it follows that  $\overline{\Omega}|_{\langle \overline{\Omega}^{k-1}(add M) \rangle}$ is not a monomorphism.  Now  $\overline{\Omega}|_{\langle \overline{\Omega}^{k+j}(add M) \rangle}$ is a monomorphism if and only if $\overline{\Omega}|_{\langle \overline{\Omega}^{j}(add \Omega^{k}(M)) \rangle}$ is a monomorphism, from this we deduce that  $\phi(M) = \phi(\Omega^{k}(M))+k$.
\end{proof}
\end{prop}

\begin{coro}\label{indescomponibles}
Let $M =\oplus_i ^t M_i$ be a decomposition of a basic module.  Assume that  $\Omega^{k}(M_i)$ is indecomposable for  any $i$ and that  $rk(\langle \overline{\Omega}^{k}(add M) \rangle ) < rk(\langle \overline{\Omega}^{k-1}(add M) \rangle )$, then $\phi(M) = \phi(\Omega^{k}(M))+k$.
\end{coro}

\section{Radical square zero algebras}

In this section  we will consider only radical square zero algebras.

\subsection{The function $\phi$ for radical square zero algebras}

\begin{defi}
Given a finite quiver $Q$, with $n$ vertices, we will call \textbf{heart} of $Q$  the full subquiver of  $Q$  determined by the vertices whose associated simple modules are in the intersection of the support of  $\Omega^n(A_0)$ with the support of
$\Omega^{-n}(A_0)$.
\end{defi}

\begin{defi}
We will call \textbf{member} of $Q$ the full subquiver of $Q$ determined by the vertices which are not in the heart.
\end{defi}

\begin{nota}
 The heart of $Q$ will be denoted by \textbf{$C(Q)$}  and the member by \textbf{$M(Q)$}.
\end{nota}

\noindent Next we present an useful and easy to prove result, a characterisation of $C(Q)$.

\begin{prop}
The heart of $Q$,
$C(Q)$, is the full subquiver of $Q$ determined by the vertices whose associated simple modules are in the support of $\displaystyle \bigcap_{t\in\Z}\Omega^t({A_0})$.
\end{prop}

\begin{ej}
\emph{Let $Q$ the following quiver}

$$ 
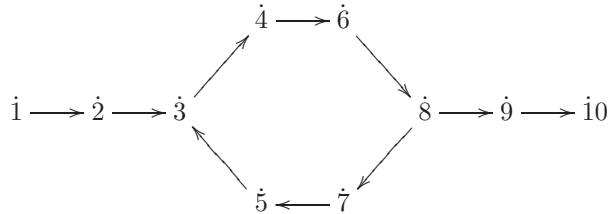
\xymatrix@=7mm{ & & & \dot 4 \ar[r] & \dot 6 \ar[rd]& & & \\
\dot 1 \ar[r] & \dot 2  \ar[r]& \dot 3  \ar[ru] & & & \dot 8 \ar[ld]\ar[r] & \dot 9 \ar[r] & \dot 10 \\
& & & \dot 5  \ar[lu]& \dot 7 \ar[l] & & &}$$

\emph{\noindent then}
$$\xymatrix@=5mm{  M(Q)= & \dot 1 \ar[r] & \dot 2 & & \dot 9 \ar[r]& \dot 10\\}
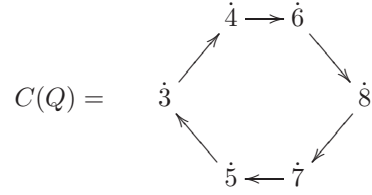
$$

$$
 \xymatrix@=5mm{ & & & \dot 4 \ar[r] & \dot 6 \ar[rd]& \\
& C(Q)= & \dot 3 \ar[ru]  & & & \dot 8 \ar[ld] \\
& & & \ar[lu] \dot 5 & \dot 7 \ar[l] & }
$$

\end{ej}

\begin{prop}\label{member}
The following statements are equivalent:
\begin{enumerate}
\item $M(Q)=\emptyset$.
\item $C(Q)=Q$.
\item $Q$ has no sinks nor sources.
\item $\mod A$ has no simple injective nor simple projective objects.
\end{enumerate}
\begin{proof}The equivalence between items 1 and 2 and between 3 and 4 are clear.\\
We show next that 3 implies 2. If $Q$ has no sinks nor source, then for each vertex $v$ there is at least one arrow starting at it and at least one arrow ending at it. Since we can construct paths of length
$n = \# Q_0$  ending and beginning at any vertex we have $C(Q) = Q$.\\
Finally we show that 1 implies 4. Assume there is an injective simple module $S_{i}$, then this simple cannot be a direct summand of the syzygy of any module.
In particular it is not in the support of $\Omega^n(A_0).$
The case of existence of a simple projective is analogous. \end{proof}
\end{prop}

\begin{defi}\label{subcorazonfinal defi}
 A  full subquiver $Q'$ of a quiver  $Q$ is called a \textbf{ final subheart} if it is minimal in the family of full subquivers of  $Q$  which are closed by sucessors.\\
 Analogously we define an \textbf{initial subheart} of $Q$ as a full  subquiver $Q'$, which  is minimal in the family of full subquivers of $Q$ which are closed by predecessors.
\end{defi}

\noindent We state the following lemma, whose proof we leave to the reader.
\begin{lema}\label{proptriviales}
The following statements hold:
\begin{enumerate}
\item Final and initial subhearts are connected.
\item If an initial or final subheart $Q'$ has at least one arrow then $Q'$ is contained in the heart of $Q$.
\item If a vertex $v$ is a sink then the subquiver consisting of the vertex $v$ is a final subheart.
 \item  The quiver $Q$ has a proper final subheart if and only if there exist a proper full subquiver $Q'$ such that $\Omega(\oplus_{i\in Q_0'} S_i) \in \add(\oplus_{i\in Q_0'} S_i)$.
  \item If  $Q'$ is a final subheart then all arrows starting at a vertex of  $Q'$ ends in $Q'$.
\end{enumerate}
The last three statements may also be stated in their dual version.
\end{lema}

\noindent We give now an example.

\begin{ej}

\emph{Let $Q$ be the following quiver:}\\
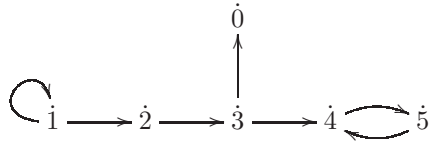
$$\xymatrix{ & &\dot 0 & \\
\dot 1 \ar@(l,u) \ar[r]& \dot 2 \ar[r]& \dot 3\ar[u] \ar[r]& \dot 4 \ar@/^/[r] & \dot 5 \ar@/^/[l]  }$$\\

\noindent \emph{The subquivers:} $\xymatrix{\dot 0 & and &\!\!\!\!\!\!\! \dot 4 \ar@/^/[r] & \dot 5 \ar@/^/[l] }\quad $ \emph{final subhearts.}
\vskip .5cm

\noindent \emph{On the other hand the subquiver} $Q'= \xymatrix{ \dot 1 \ar@(r,u)}$ \ \ \emph{is an initial subheart.}
\end{ej}

\vskip .2cm

\noindent For radical square zero algebras we show the following result concerning the finitistic dimension.

\begin{lema}\label{findim0}
Let $A$ be a non simple radical square zero algebra. Then $Q$ does not have a sink  if and only if $\fin(A) = 0$.

\begin{proof}
Assume $Q$ has a sink $v$. Since $A$ is not simple and connected $v$ is not a source. Let $w$ be an immediate predecessor of $v$. We can construct a non-split short exact sequence as follows:

 $$\eta_{S_v}:\quad\quad \xymatrix{ 0 \ar[r]& S_{v} \ar[r]& P_{w}=P(M_{S_v})  \ar[r]& M_{S_v} \ar[r]& 0}$$

\noindent Therefore   $\Omega(M_{S_v}) = S_v$, and it follows that  $\pd (M_{S_v}) = 1$.\\
We show now the other implication. If $Q$ does not have a sink  then there is no simple projective module. Since the syzygy of any module is semisimple it follows that the projective dimension of any non projective module is infinite, so $\fin(A)=0$.
\end{proof}
\end{lema}

\begin{obs}\label{remarkita}
\noindent Assume $S=S_v$ is the  simple module associated to a non source vertex $v\in Q_0$. We will use $\eta_{S}$, and $M_S$ to describe the sequence and any module given as in the proof of the last lemma.
\end{obs}

\begin{prop}\label{cotarad2}
If $A$ is a radical square zero algebra then:  $$\fidim(A) \leq \phi(A_0) + 1.$$

\begin{proof}
Let $M \in \mod A$. We get $\phi(M) \leq \phi(\Omega(M))+1$ from Lemma 3.4 of [HLM]. Moreover $\Omega(M) \in \add A_0$, so using Lemma \ref{cuentita} we see that  $\fidim(A)\leq \phi(A_0) + 1$.
\end{proof}
\end{prop}
\begin{prop}\label{autoinyectiva} For a connected non simple radical square zero algebra $A$ the following statements are equivalent:
\begin{enumerate}
\item  $\mathcal{S} = \mathcal{S}_D$ and for each simple module $S \in \mathcal{S}$, the module $M_S$ is simple(see \ref{remarkita}).
\item $Q$ is a cycle ($Z_n$).
\item $A$ is a Nakayama algebra without injective simple modules.
\item $A$ is a Nakayama algebra without projective simple modules.
\item $A$ is a selfinjective algebra.
\item All indecomposable projective modules have length 2.
\item All indecomposable injective modules have length 2.
\item The $\phi$-dimension of $A$ is zero.
\end{enumerate}
\begin{proof}
We start by showing that the  first statement implies the second one.\\
Since we are assuming $\mathcal{S} = \mathcal{S}_D$ the family of exact sequences $\{\eta_S\}_{S\in\mathcal{S}}$ shows that $l(P({M_S}))= 2$ for all $S\in \mathcal{S}$. But in this case $\{M_S\}_{S\in\mathcal{S}}$ is the family of all simple modules.  Therefore for each vertex there is exactly one arrow starting at it.  So the number of arrows is the same as the number of vertices of the quiver.
Since there is no source, for $\mathcal{S} = \mathcal{S}_D$, we have that for any vertex there is at least one arrow ending at it. Since the number of arrows and vertices coincides, there is exactly one arrow ending and one arrow starting at any vertex.
Because the quiver $Q$ is finite we get $Q$ is a cycle $Z_n$.\\
It is clear that the second statement implies all the other statements.\\
We now show that third statement implies the second one. Follows from the fact that the quiver of a Nakayama algebra is either a linearly ordered $A_n$ or a cycle $Z_n$. If there is no simple injective it must be a cycle.\\
Equivalence between two and four is similar to the equivalence between two and three.\\
Finally we show that the fifth statement implies the second one. The hypothesis implies that there is exactly one arrow starting and one arrow ending at any vertex. Therefore $Q$ is a cycle $Z_n$.
We leave to the reader to show that the other statements are also equivalent to these ones. \end{proof}
\end{prop}


\noindent  In the proof of next proposition we construct, for a radical square zero non selfinjective algebra, a witness module, that is a module $M\in \mod A$ with $\phi(M)$ equal to the $\phi$-dimension of the algebra.

\begin{prop}\label{explicito}
If $A$ is not selfinjective then  $\fidim(A) = \phi(\oplus_{S\in \mathcal{S}_D}S)+1$.

\begin{proof}
For each module $M$ we have that $\phi(M)\leq \phi(\Omega M)+1\leq \phi(\oplus_{S\in \mathcal{S}_D}S)+1$. So $\fidim(A)\leq \phi(\oplus_{S\in \mathcal{S}_D}S)+1$. Now for each $S\in \mathcal{S}_D$, consider an indecomposable module $M_S$ such that  $\Omega(M_S) = S$. If some $M_S$ is not simple, for some simple $S\in \mathcal{S}_D$, by Proposition \ref{epimorfismo} with $k=1$, we get that $\phi((\oplus_{\mathcal{S}_D}  M_S)\oplus (\oplus_{ \mathcal{S}_D}S) ) = \phi(\oplus_{ \mathcal{S}_D} S) +1$, and the result follows. If $M_S$ is simple for all $S \in \mathcal{S}_D$ since the algebra $A$ is not selfinjective we can use the Proposition \ref{autoinyectiva} to see that there is a  simple  module $S_0$ which is projective or injective.

\noindent Suppose $S_0$ is an injective (no projective) module. If $S_0$ is not isomorphic to $M_S$ for any $S\in \mathcal{S}_D$ then, by Proposition \ref{epimorfismo} with $k=1$, $\phi((\oplus_{\mathcal{S}_D} M_S )\oplus S_0) = \phi(\oplus_{\mathcal{S}_D} S)+1$. If $S_0$ is isomorphic to $M_S$, for some $S\in \mathcal{S}_D$, then there is a module $S_1 \in \mathcal{S}_D$ which is not isomorphic to any of the modules $M_S$. Then, again by Proposition \ref{epimorfismo} with $k=1$, we get that $\phi((\oplus_{\mathcal{S}_D} M_S )\oplus S_1) = \phi(\oplus_{\mathcal{S}_D} S)+1$. Finally if  $S_0$ is projective (non injective), then $\Omega(M_{S_0}) = S_0$  and as a consequence of Corollary \ref{indescomponibles} we see that $\phi((\oplus_{\mathcal{S}_D} M_S) \oplus M_{S_0}) = \phi(\oplus_{\mathcal{S}_D} S)+1$, and the result follows. \end{proof}
\end{prop}

\begin{coro}\label{sinmember}
If  $A$ is not selfinjective and $M(Q)$ is empty then $\fidim(A) = \phi(A_0)+1$.
\end{coro}
\begin{prop}\label{cotazocalo}
 If $Q$ has $n$ vertices then $\phi(A_0)\leq n-1$.
\begin{proof}
Since  $\Omega(A_0) \in \add A_0$, we can use  Proposition \ref{invariante} to obtain that  $\phi(A_0) \leq n $. If $\phi(A_0) = n$, again by Proposition \ref{invariante},  we see that $\pd(A_0)=n$ is finite, and so $\gd(A)=\pd(A_0)=n$ is finite. But for a radical square zero algebra with finite global dimension it holds that $\gd(A)\leq \# Q_0 - 1$ (see Remark \ref{fgd}), a contradiction. \end{proof}
\end{prop}

\begin{coro}\label{toto}
If $A$ is a radical square zero algebra then  $\fidim(A)\leq n.$

\begin{proof}
This is a consequence of Propositions  \ref{cotazocalo}  and \ref{cotarad2}.
\end{proof}
\end{coro}

\noindent We present now and example of an algebra for which the previous inequality can be strict.

\begin{ej}\label{1pto} 
\emph{Let be $A$ a radical square zero algebra with $Q$ given by}
 $$Q = \xymatrix{1 \ar@(d,r)}$$\\
\emph{Since this is a selfinjective algebra  then $\fidim(A) = 0$, (Theorem 5 of [HL]).}
\end{ej}


\noindent Let $k=\# ({\mathcal{S}_P} \cup {\mathcal{S}_I})$.

\begin{prop}\label{phimembermaximal}
If $A$ has infinite global dimension and $M(Q)$ is not empty then $\fidim(A) \leq \# {\mathcal{S}_D}= n-k$.

\begin{proof}
Let $\{S_1, \dots, S_k\}$ be the set of isomorphism classes of simple modules in $S_P\bigcup S_I$, and $\{S_{k+1},\dots, S_n\}={\mathcal{S}_D}$ a basis of $K_1$. Invoking  Proposition \ref{invariante} item 3, we see that  $\displaystyle\phi(\oplus_{i = k+1}^n S_{i}) \leq n-k-1$.
\end{proof}
\end{prop}

\noindent The Corollary \ref{toto} and Proposition \ref{phimembermaximal} suggest the following definition.

\begin{defi}\label{maximalphidimension}
We say that a radical square zero algebra $A$ with $n$ vertices has maximal $\phi$-dimension if $\phi(A)= n$, in case $M(Q)$ is empty, or $\phi(A)= n-1$, in case $M(Q)$ is not empty.
\end{defi}

\subsection{Left and right $\phi$-dimension}

It is well known that for a noetherian ring the left global dimension is equal to the right global dimension (see [A1]). We do not know any example of an Artin algebra $A$ where the left $\phi$-dimension ($\fidim_l(A)= \fidim_r(A^{op})$) is not equal to the right  $\phi$-dimension of $A$ ($\fidim_r(A)$). We prove now that they are equal if $A = \frac{\mathbb{K} Q}{J^2}$ by showing that $\fidim_r(A)= \fidim_r(A^{op})$.

\begin{teo}\label{izq=der}
If $A$ is a radical square zero algebra then $\fidim (A) = \fidim(A^{op})$.
\end{teo}
\begin{proof}
We know by \cite{kn:marchuard} that $\fidim (A) = 0$ if and only if $A$ is selfinjective, and this is equivalent to $A^{op}$ be selfinjective. Again by  \cite{kn:marchuard} the last is equivalent to $\fidim (A^{op})= 0$. We assume now that $\fidim (A) > 0$ (so $\fidim (A^{op}) > 0$). Let $T= \oplus_{S\in \mathcal{S}_D}S$, and $\tilde{T}= \oplus_{\tilde{S}\in \tilde{\mathcal{S}}_D}\tilde{S}$ where $\tilde{\mathcal{S}}_D$ are the simple modules of $A^{op}$ which are not injective nor projective. Using Proposition \ref{explicito} we see that $\fidim (A) = \phi(T)+1 $ and $\fidim (A^{op}) = \phi(\tilde{T})+1 $.
Let $ \beta = \{S\}_{S \in \mathcal{S}_D}$ and $ \beta' = \{\tilde{S}\}_{\tilde{S} \in \tilde{\mathcal{S}}_D}$ be the basis of the subgroups $<add T>$ and $<add \tilde{T}>$ in $K_0(A)$ and $K_0(A^{op})$ respectively.

\noindent Let $\mathfrak{M}$ be the matrix of the endomorphism $\bar{\Omega}|_{<addT>}:<add T>\to <add T>$ in the basis $\beta$ and $\mathfrak{N}$ the matrix of $\bar{\Omega}|_{<add\tilde{T}>}:<add \tilde{T}>\to <add \tilde{T}>$ in the basis $\tilde{\beta}$.
By Proposition \ref{invariante} we get that $\phi(T) = min\{l: rk(\mathfrak{M}^l) = rk(\mathfrak{M}^{l+1})\}$ and
$\phi(\tilde{T}) =min\{l: rk(\mathfrak{N}^l) = rk(\mathfrak{N}^{l+1})\}$. The result follows from the fact that $\mathfrak{M}$ is the transpose of $\mathfrak{N}$.
\end{proof}

\subsection{Quivers with members}
We have shown that if $M(Q)$ is not empty then $\fidim(A) \leq n-1$. We say that an algebra with $M(Q)$ not empty and $\fidim = n-1$ is an algebra with maximal $\phi$-dimension.

\begin{ej}\label{ejemplito}
\emph{The radical square zero algebras whose quiver are of the form $Q$ or $Q'$ given below have no empty members and their $\phi$-dimension is equal to $n-1$,}

\begin{itemize}
\item $ Q = A_n = \xymatrix{  1 \ar[r]& 2 \ar[r] & 3 \ar[r]& \ldots \ar[r]& n }$

\emph{and}

\item $  Q'= \xymatrix{  1 \ar[r]& 2 \ar[r] & 3 \ar[r]& \ldots \ar[r]& n \ar@(u,r)}$.
\end{itemize}
\emph{\noindent Observe that the algebras on the first family have finite global dimension and the ones in the second family have infinite global dimension with finitistic dimension equal to zero.}
\end{ej}

\noindent We define now a pre-order relation on the set of vertices  $Q_0$:

\begin{defi} We define in $Q_0$ the following preorder relation: given $x,y\in Q_0$ then $x\preceq y$ if there is a path in $Q$ starting at $x$ and ending at $y$.
\end{defi}

\begin{ej}
\emph{Let}  $Q = \xymatrix{1 \ar@/^/[r]& 2 \ar@/^/[l] \ar[r] & 3 }$. \emph{Then  $1 \preceq 2$ and $2 \preceq 1$, therefore in this example $(Q_0,\preceq)$ is a preorder which is not an order.}
\end{ej}

\begin{obs}
The preorder $(Q_0,\preceq)$ is an order if and only if every cycle in $Q$ is a composition of loops.
\end{obs}

\begin{lema}
The restriction $(M(Q)_0,\ \preceq )$ is an order.
\end{lema}
\begin{proof}
This is a consequence of the fact that the vertices of  $M(Q)$ are not in cycles of $Q$.
\end{proof}

\begin{obs}
The member is completely ordered, with the order given above, if and only it has exactly one sink and one source if and only if there is  a path on the member of length $k-1$, where $k$ is the number of vertices on the member.
\end{obs}

\noindent Using the remark above we get the following proposition.

\begin{prop}\label{member ordenado}

If $M(Q)$ is not empty and $\fidim (A)=n-1$ ($A$ has maximal $\phi$-dimension) then $(M(Q)_0,\ \preceq)$ is a total order.
\end{prop}

\begin{proof}
\noindent We have to show that if  $v_1$, $v_2$ are two distinct vertices on the member, then there is a path from $v_1$ to $v_2$ if and only if there is no path from  $v_2$ to $v_1$. Since the former lemma states that the relation is an order on $M(Q)$ we already have that, if there is a path from  $v_1$ to $v_2$ then there is no path from $v_2$ to $v_1$. Next we show that in case there is no path from $v_1$ to $v_2$ then there is a path from $v_2$ to $v_1$.
\begin{enumerate}
  \item Assume first that $A$ has finite global dimension. In this case the quiver of $A$ has a path of length $n-1$, thus the order is total.
  \item It remains to be considered the algebras $A$ with infinite global dimension. In this case the number of sources added with the number of sinks is 1.\\
 Consider the case where there is only one source $v_0$, and no sink.
         Assume there are vertices $v_1$ and $v_2$ in  $M(Q)$ which are not connected by a path. If  $M(Q)$ has $k$ vertices,  the largest path in $M(Q)$ which does not starts in $v_{0}$, has length at most $k-3$ (otherwise the order on the member would be total). This implies that $\Omega^{k-2}(\oplus_{S \in \mathcal{S}_D}S)$ is a direct summand of  $\oplus_{w \in  C(Q)}S_w$.\\
    Applying, various times, Lemma 3.4 of  \cite{kn:hlm1} we get the following inequality:
       $$\phi(\oplus_{S \in \mathcal{S}_{D} }S) \leq \phi(\Omega^{k-2}(\oplus_{S \in \mathcal{S}_{D}}S)) + k-2$$
       Consider now the radical squared zero algebra $\Gamma$ determined by the heart  $ C(Q)$,  ($\Gamma = \frac{\mathbb{K}  C(Q)}{J^2}$). We obtain, using Proposition \ref{cotazocalo}, that  $\phi_{\Gamma}(\oplus_{ w\in  C(Q)} S_{w})$ $\leq n-k-1$. We see that  $A$ is an algebra which can be put in the triangular form below:
       $$A = \left(
                                \begin{array}{cc}
                                  S & M \\
                                  0 & \frac{\mathbb{K}  C(Q)}{J^2} \\
                                \end{array}
                              \right).$$

\noindent  and using the Remark  \ref{extensionconocida} we get $\phi_{A}(\oplus_{ w\in  C(Q)} S_{w}) = \phi_{\Gamma}(\oplus_{ w\in  C(Q)} S_{w})$. So we obtain the expression:

       $$\phi_{\Gamma}(\oplus_{S \in \mathcal{S}_{D} }S) \leq n-k-1 + k-2 = n-3.$$

        Using the Proposition  \ref{explicito},  we get $\phi \dim (A) = n-2$,  which is a contradiction with the hypothesis.\\
Finally, the case with one sink and no source, is reduced to the former case using Theorem \ref{izq=der}.\end{enumerate}
\vspace{-0.55cm}
\end{proof}

\subsection{Quivers without member}
We recall that the $\phi \dim (A)$ is at most the number of vertices of $Q$, and equality implies that  the quiver of $A$ does not have a member. We will exhibit in the Example \ref{phimaximal} an infinite family of algebras without member and maximal $\phi$-dimension, that is an infinite family of algebras with $\phi$-dimension equal to the number of vertices. \\
In the case of an algebra $A$ without member we have that the adjacency matrix of the quiver $Q$ coincides with the associated matrix of the linear transformation $\overline{\Omega} |_{K_1}$ in the base  $\{[S]\}_{S\in \mathcal{S}}$ of  $K_1$. We will say that $A$ has algebraic multiplicity and geometric multiplicity $m_a$ and $m_g$ respectively if these are the corresponding multiplicities of the adjacency matrix. \\
The following remark is useful to compute the  $\phi$-dimension:

\begin{obs}\label{lineal}
Considering the   $\mathbb{C}$-linear map:
$$\overline{\Omega}{|}_{\langle \add M\rangle}\otimes_{\mathbb{Z}} \mathbb{C}: \mathbb{C}^{rk(\langle \add M\rangle)}\rightarrow \mathbb{C}^{rk(\langle \add M\rangle)}$$

\noindent we have $rk (\overline{\Omega}{|}_{\langle \add M\rangle}\otimes_{\mathbb{Z}} \mathbb{C}) = rk(\overline{\Omega}{|}_{\langle \add M\rangle})$.
\end{obs}

\noindent In the next example we exhibit a family of radical square zero algebras with $n$ vertices such that $\phi \dim(A(n))= n$.

\begin{ej}\label{phimaximal}
\emph{Consider a radical square zero algebra whose quiver has the adjacency $\mathfrak{M}_n$ given below:}
$$\mathfrak{M}_n=\left(
  \begin{array}{ccccccc}
    1 & 0 & 0 & 0 & \ldots & 0 & a_1 \\
    1 & 1 & 0 & 0 & \ldots & 0 & a_2 \\
    0 & 1 & 1 & 0 & \ldots & 0 & a_3 \\
    0 & 0 & 1 & \ddots &  & \vdots & \vdots  \\
    \vdots & \vdots &  & \ddots & 1 & 0 & a_{n-2}\\
    0 & 0 & 0 & \ldots & 1 & 1 &a_{n-1} \\
    0 & 0 & 0 & \ldots & 0 & 1 & a_n \\
  \end{array}
\right)$$

\emph{\noindent A vector  $(x_1, x_2, \ldots, x_n)$  in the kernel of the linear transformation verify the following equations:}

\begin{itemize}
  \item $x_1 + a_1x_n = 0$,
  \item $\vdots$
  \item $x_j + x_{j+1} + a_{j+1}x_n=0$,
  \item $\vdots$
  \item $x_{n-2} + x_{n-1}+ a_{n-1}x_n=0$,
  \item $x_{n-1} + a_nx_n =0$.
\end{itemize}
\emph{\noindent Which means that we have:}

\begin{itemize}
  \item $x_j = (\sum_{i = 1}^j (-1)^{j+i-1}a_i)x_n$ para $j = 1, 2 \ldots, n-1$,
  \item $x_{n-1} = -a_nx_n$.
\end{itemize}
\emph{\noindent Observe that the first two columns of $(\mathfrak{M}_n)^{n-2}$ are:\\}

$$\begin{array}{cc}
  \left(
     \begin{array}{c}
       C_0^{n-2} \\
       C_1^{n-2} \\
       C_2^{n-2} \\
       \vdots \\
       C_{n-3}^{n-2} \\
       C_{n-2}^{n-2} \\
       0 \\
     \end{array}
   \right)
   & \left(
        \begin{array}{c}
          0 \\
          C_0^{n-2} \\
          C_1^{n-2} \\
          \vdots \\
          C_{n-4}^{n-2} \\
          C_{n-3}^{n-2} \\
          C_{n-2}^{n-2} \\
        \end{array}
      \right)
\end{array}
$$\\

\emph{\noindent We require that  $Ker\ \mathfrak{M}_n \subset Im((\mathfrak{M}_n)^{n-2}$,  and this can be obtained when the kernel is constructed from the two vectors above. Since  $C_{0}^{n-2} = C_{n-2}^{n-2} = 1$  then $a_j$ satisfy the following conditions:}
\begin{itemize}
    \item $a_2 = a_1(C_1^{n-2}+1) - C_0^{n-2}$
    \item $\vdots$
    \item $a_j = a_1(C_{j-1}^{n-2}+C_{j-2}^{n-2}) - (C_{j-2}^{n-2}+C_{j-3}^{n-2})$
    \item $\vdots$
    \item $a_n = a_1 - (n-2)$.
\end{itemize}
\emph{So  we take $a_1> n-2$, then each  $a_j$ with $j = 2, \ldots, n$ will be positive and we have that $\mathfrak{M}_n$ is an adjacency matrix of a quiver $Q$, such that:}

\begin{enumerate}
  \item \emph{It has no member,}
  \item $\phi(A_0) = n-1$
\end{enumerate}
\emph{\noindent Finally we get that $\fidim(A(n)) = n$.}
\end{ej}

\begin{coro}
Let $A$ be a radical square zero algebra with $\phi$-dim$(A)=n$. Then $Q$ has no member.
\end{coro}
\begin{proof}
This is clear using Proposition \ref{phimembermaximal} and considering the family of the Example \ref{phimaximal}.\end{proof}

\noindent Using the Proposition  \ref{lineal} we get the following result:

\begin{teo}\label{forma de jordan 2}
Let $A$ be a radical square zero algebra without member and  $n$ vertices, with $n \geq 2$. Then $\phi \dim(A) = m+1$ where $m$  is the size of the biggest Jordan block associated to $0$.\end{teo}
\begin{proof}
Let  $\{v_1, \ldots, v_m, v_{m+1}, \ldots,v_{m+k}, v_{m+k+1}, \ldots, v_n\}$ be a Jordan basis for the adjacency matrix $\mathfrak{M}$ of $A$,
where the vectors $\{v_1, \ldots, v_{m+k}\}$ are those which determines a base of the proper subspace associated  with  $0$. We fix $\{v_1, \ldots, v_{m}\}$  as the vectors associated to the  block of size  $m$ and the  vectors $\{v_{m+k+1}, \ldots, v_n\}$  are those which form a base of the proper subspaces associated with the non zero eigenvalues.\\
Since $\mathfrak{M}^{m+l}(v_i) = 0$ for each  $l \geq 0$ and each $1 \leq i \leq m+k$, then  $Im\ \mathfrak{M}^{m+l} = \langle\{ v_{m+k+1}, \ldots, v_n\}\rangle $ for each $l\geq 0$. Moreover since $\mathfrak{M}^{m-1}(v_1) = v_m$ and $\mathfrak{M}$ restricted to $Im\ \mathfrak{M}^{m+l}$ is injective for all  $l \geq 1$, we see that  $\phi(A_0) = m$, and then, using Corollary \ref{sinmember}, we get that $\phi \dim (A) = m+1$. \end{proof}

\begin{obs}\label{traza}
The trace of the adjacency matrix of a quiver $Q$ is equal to the number of loops.
\end{obs}

\begin{teo}\label{forma de jordan}
Let  $A$ be a radical square zero algebra with with $n$ vertices and $n \geq 2$, then  $\fidim(A) = n$  if and only if $A$ has no member and the Jordan form of the adjacency matrix of $Q$, $\mathfrak{M}_Q$, has the form:

$$\left(
    \begin{array}{ccccccc}
      0 & 0 & \ldots & 0 & 0 & 0 & 0\\
      1 & 0 & \ldots & 0 & 0 & 0 & 0\\
      0 & 1 &  & 0 & 0 & 0 & 0\\
      \vdots &  & \ddots & \ddots & \vdots & \vdots &\vdots \\
      0 & 0 & \ddots & 1 & 0 & 0 & 0\\
      0 & 0 & \ldots & 0 & 1 & 0 & 0\\
      0 & 0 & \ldots & 0 & 0 & 1 & \lambda\\
    \end{array}
  \right)
$$
\noindent where $\lambda$ is not zero (in this case $\lambda$ is equal to the number of loops).
\begin{proof} Assume first that the Jordan form of the adjacency matrix of $Q$ has the form above. Then $\phi(A_0)= n-1$, therefore by  Corollary \ref{sinmember} we get $\fidim(A) = n$.\\
Next we show the other implication. \\
So we assume that $\fidim(A) = n$. It is clear that  $0$ is an eigenvalue of $\mathfrak{M}_Q$,  otherwise we would have   $\phi(A_0) = 0$ which would imply $\fidim(A) \leq 1 < n$.  It is also clear that there is a non zero  eigenvalue  $\lambda$, otherwise   $\gl(A) < \infty$ and using the Remark \ref{fgd} we would get  $\fidim(A) = \gd(A)< n$. Moreover, there is no other non zero eigenvalue distinct from  $\lambda$  and the algebraic multiplicity of $\lambda$ has to be $1$ since if this were not the case we would have $\phi(A_0)<n-1$.\\
We know that there is a Jordan block of size $1\times 1$  associated with the eigenvalue $\lambda$ and a Jordan block associates with the eigenvalue  $0$. If the Jordan block associated with  $0$ would have various sub blocks we would get that  $\phi(A_0)$ would not be $n$. So the Jordan form of $\mathfrak{M}_Q$ must be:
$$\left(
    \begin{array}{ccccccc}
      0 & 0 & \ldots & 0 & 0 & 0 & 0\\
      1 & 0 & \ldots & 0 & 0 & 0 & 0\\
      0 & 1 &  & 0 & 0 & 0 & 0\\
      \vdots &  & \ddots & \ddots & \vdots & \vdots &\vdots \\
      0 & 0 & \ddots & 1 & 0 & 0 & 0\\
      0 & 0 & \ldots & 0 & 1 & 0 & 0\\
      0 & 0 & \ldots & 0 & 0 & 1& \lambda\\
    \end{array}
  \right)\ $$\end{proof}
\end{teo}

\vspace{5mm}

\noindent We look now for necessary conditions in order to have the $\phi$-dimension equals to the number of vertices.

\begin{defi}
Given a quiver $Q$, we call \textbf{starting degree (arriving degree)} of a vertex $v$ the number of arrows starting (arriving)  at   $v$.  The quiver $Q$ will be called  \textbf{regular at starting (arriving)} if the starting (arriving) degree  is the same for all vertices of  $Q$, in this case we call this common number the \textbf{starting (arriving) degree of $Q$}.
\end{defi}

\noindent The following proposition give us a necessary condition to have maximal $\phi$-dimension in quivers with regular starting degrees.

\begin{prop}
If $Q$ has $n$ vertices, no member and it is regular at starting then  $\phi \dim (A)=n$ implies that the number of loops is equal to the starting degree of $Q$.
\begin{proof} Let $\mu$ be the starting degree of  $Q$. It is clear that the  vector $(1, \ldots, 1)$ is an eigenvector associated to the eigenvalue $\mu$. If the number of loops do not coincides with the starting degree we would have two eigenvectors associated to eigenvalues which are not zero. This contradicts Theorem  \ref{forma de jordan}.
\end{proof}
\end{prop}

\noindent We give next a necessary condition  in order to have the $\phi$-dimension equals to the number of vertices in terms of final (initial) subhearts. We show that a radical square zero algebra $A$ which associated quiver $Q$ has proper final (initial) subhearts does not get maximal $\phi$ dimension.

\begin{prop}\label{subcorazon final}
If $Q$ has no member and there is a proper final (initial) subheart $Q'$, then $\fidim(A) < n$.

\begin{proof} Assume that the final subheart  $Q'$ has $k$ vertices ($\# Q'_0 = k$)  and let $Q''$ be the full subquiver whose vertices are  those of $Q_0\setminus Q'_0$. Because there are no paths from $Q'$ to  $Q''$, since $Q'$ is a final subheart  of $Q$, the adjacency matrix $\mathfrak{M}_Q$ of $Q$ has the following form:
$$\left(
    \begin{array}{cc}
      \mathfrak{M}_{Q''} & \mathfrak{L} \\
      0 & \mathfrak{M}_{Q'} \\
    \end{array}
  \right)
$$

\noindent where  $\mathfrak{M}_{Q''}$ and $\mathfrak{M}_{Q'}$ are the matrices of $Q''$ and $Q'$, respectively. Therefore the
characteristic polynomial $p_Q(x)$ of  $Q$ is such that  $p_Q(x) = p_{Q'}(x)p_{Q''}(x)$ where $p_{Q'}(x)$ and $p_{Q''}(x)$ are the characteristic polynomials of $Q'$ and $Q''$ respectively.
\noindent Because $Q$ has no member, for each vertex $v$ of $Q''$ there are paths with arbitrary length arriving at $v$. Since there is no paths from $Q'$ to  $Q''$, we get that $\mathfrak{M}_{Q''}^{n-k}\neq 0$, so  $p_{Q''}$ has a nonzero  root $\lambda$.  Since $Q'$ is a final subheart we get that $\mathfrak{M}_{Q'}^k\neq 0$ which implies that $p_{Q'}$ also has a nonzero root $\mu$.  Finally we have that  $p_{Q}$ has two roots (if  $\lambda \neq \mu$) or just one nonzero root with multiplicity at least $2$ (in case  $\lambda = \mu$). Using the Theorem \ref{forma de jordan} we get that  $\fidim(A)<n$. \end{proof}
\end{prop}

\begin{defi}
We say that a quiver $Q$  is strongly connected if any pair of vertices are connected by a path.
\end{defi}

\begin{obs} \label{strongly connected}
A quiver $Q$ is strongly connected if and only if $Q$ has neither proper initial nor proper final subhearts.
\end{obs}

Using Proposition \ref{subcorazon final} and Remark \ref{strongly connected} we get the following result:

\begin{coro} \label{sucorazon final fuertemente conexo}
If $\fidim (A)=n$ then  $Q$ is strongly connected.
\end{coro}

The following example shows that the reciprocal of Corollary \ref{sucorazon final fuertemente conexo} is not valid.

\begin{ej}

\emph{Consider $Q$ the following strongly connected quiver:\\}

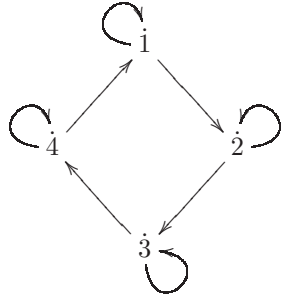
$$ \xymatrix{ & \dot 1 \ar[rd] \ar@(l,u) & \\
               \dot 4 \ar[ur]  \ar@(l,u) &   & \dot 2 \ar[dl]  \ar@(r,u)\\
              &  \dot 3 \ar[ul]  \ar@(d,r)& }$$\\

\emph{Its adjacency matrix $M(Q)$ is:\\}

$$\left(
  \begin{array}{cccc}
    1 & 0 & 0 & 1 \\
    1 & 1 & 0 & 0 \\
    0 & 1 & 1 & 0 \\
    0 & 0 & 1 & 1 \\
  \end{array}
\right)$$

\emph{\noindent Therefore $\phi (S_1 \oplus S_2\oplus S_3 \oplus S_4) = 1$ and by Corollary \ref{sinmember} we get that  $\fidim(A) = 2$.}
\end{ej}

\begin{obs}\label{Ck}
Consider an algebra $A$ such that  $\fidim (A) = n$ where $n=|Q_0|$. Let  $C^{v}_k$ be the number of simple
cycles of length  $k$ which pass through $v$ and set $C_k = \sum_{v \in Q_0}{C^{v}_k}$ and $M_Q$ the adjacency matrix of  $Q$. The following statements are valid:

\begin{itemize}
  \item $tr(M) = C_1 = \sum_{v \in Q_0}{C^{v}_1}$
  \item $tr(M^2) = \sum_{v \in Q_0}{C^{v}_1}^2+2C_2$
  \item $tr(M^l) = \sum_{v \in Q_0}{C^{v}_1}^{l}+lC_l+H$ where $H$ is a nonzero natural number.
\end{itemize}

\end{obs}
We get as consequence the following corollary.

\begin{coro}

If $\fidim(A) = n$ with  $|Q_0| = n>1$ then $A$ has loops in more than one vertex.

\begin{proof}
Assume that  $A$ has loops only in the  vertex $v$, then by Remark \ref{traza} and Theorem \ref{forma de jordan} we have that:

$$C_1^2 = C_1^2 + 2C_2$$

\noindent and in general for each $l \in \mathbb{N}$ we get:

$$C_1^{l} = C_1^{l} + lC_l + H' $$

\noindent  where  $H'$ is bigger than zero. From the equations above we deduce that  $C_l = 0$ for each $l \in \mathbb{N}$, which cannot happens since  $n>1$ and $\fidim (A)=n$ with $n\geq 2$.

\end{proof}

\end{coro}

\subsection{Quotient by equitable partition}

We will use now the concept of equitable partition (see
 \cite{kn: GoR} Chapter $9.3$) to obtain results with respect to the maximality of the  $\phi$-dimension for radical square zero algebras.

\begin{defi}

Given a quiver $Q$, we say that a partition $\pi = \{C_1, \ldots, C_r\}$ of $Q_0$ is \textbf{equitable} if the number of vertices of  $C_j$ which are immediate sucessor of each vertex  $v$ of  $C_i$  is a constant $b_{ij}$.
\end{defi}

Given a quiver  $Q$  we always can define the  \textbf{trivial partition } which is given by $\pi = \{\{v\}\}_{v \in Q_0}$.
Clearly the trivial partition is equitable.

\begin{defi}

\noindent Given a quiver  $Q$ and an equitable partition  $\pi = \{C_1, \ldots, C_r\}$, we can define the \textbf{quotient quiver} $Q/\pi$ where $(Q/\pi)_0 = \{C_1, \ldots, C_r\}$ and the number of arrows from  $C_i$ to  $C_j$ is $b_{ij}$.
\end{defi}

\begin{prop}
Let $Q$ be a quiver without member, then given an equitable partition  $\pi$ we have that  $Q/\pi$ is a quiver without member.

\begin{proof} A vertex  $v$ is in the heart $C(Q)$ of a quiver $Q$ if and only if there exists a path from  $v$ to a cycle and another path from a cycle to $v$.\\

\noindent If we have a cycle $C$ in $Q$, the classes of arrows in $C$  form a cycle in $Q/\pi$, possible of smaller length. If we consider a vertex  $C_j$ in $Q/\pi$  then for each vertex in $C_j$ there  is a path to a cycle of  $Q$ and a path to it, since  $C(Q) = Q$. The previous paths give in the quotient  $Q/\pi$ a path from $C_j$ to a cycle in $Q/\pi$ and a path from a cycle in $Q/\pi$ to $C_j$ respectively.

\end{proof}

\end{prop}

\begin{obs}
Given an automorphism of quivers, the set of orbits form an equitable partition.
\end{obs}

\begin{defi}
Given a partition  $\pi =\{C_1, \ldots, C_r\}$ of the vertex set of a quiver  $Q$ we define the  \textbf{characteristic matrix} $P$ as the matrix with $|Q_0|$  rows and  $r$ columns  such that the $i^{th}$ column of
 $P$ has a $1$ in the $j^{th}$ row if  $v_j \in C_i$ and  $0$ if  $v_j \notin C_i$.
\end{defi}

\begin{obs}\label{P^tP}
We state without proof the following fact, as usual the reader can find a proof in the book \cite{kn: GoR}.
\begin{itemize}
  \item $(P^{t}P)_{ii} = |C_i|$ for all  $i$,
  \item $(P^{t}P)_{ij} = 0$ if $i \neq j$,
  \item As a consequence of the former items we have that $P^{t}P$  is an invertible matrix.
\end{itemize}

\end{obs}

 We now state some facts, adapting to our context,  some of the  results which appear in the book \cite{kn: GoR}. Since the techniques which we use are similar we do not give the proofs here, the interested reader may see  complete proofs in \cite{kn:mata}.

\begin{lema}\label{particion equitativa}
Let $Q$ be a quiver without member and $\pi$ an equitable partition of  $Q$ with characteristic matrix  $P$. If $M_Q$ and $M_{Q/\pi}$ are the adjacency matrices of  $Q$ and  $Q/\pi$ respectively, then  $M_QP =PM_{Q/\pi}$ and  $M_{Q/\pi} = (P^{t}P)^{-1}P^{t}M_QP$.\\
\end{lema}

\begin{lema}
Let $Q$ be a quiver with adjacency matrix  $M_Q$ and let $\pi$ be a partition of $Q_0$ with characteristic matrix $P$. Then $\pi$ is equitable if and only if the column space is invariant by $M_Q$.\\
\end{lema}

\begin{obs}
Another time we state a fact whose proof can be found in \cite{kn: GoR}.\\
If $AP = PB$ then for  each polynomial $f$, we have that $f(A)P = Pf(B)$. Therefore if $f(A)=0$ then  $f(B)=0$ and we conclude that the minimal polynomial of $B$ divides the minimal polynomial of $A$.
\end{obs}

\begin{teo} \label{polinomio caracteristico pi}
If $\pi$ is an equitable partition of the quiver  $Q$, and $M_Q$ is the adjacency matrix of  $Q$ and $M_{Q/\pi}$ is the adjacency matrix of $Q/\pi$. Then the characteristic polynomial of  $ M_{Q/\pi}$ divides the characteristic
polynomial of  $M_Q$.
\end{teo}
We use the results above in order to obtain the following relation between the $\phi$-dimension of an algebra and the $\phi$-dimension of a quotient by an equitable partition.

\begin{prop}\label{cota equitativa}
Let $A= \frac{\mathbb{K} Q}{J^2}$ be a radical square zero algebra without member. If $\pi$ is an equitable partition of $Q$, then  $$\phi \dim\left(\frac{\mathbb{K}(Q/\pi)}{J^{2}}\right) \leq\ min\{m, \phi \dim(A)\},$$ where $m$ is the number of vertices of  $Q/\pi$.\\

\begin{proof} The adjacency matrix of  $Q$ is similar to a matrix of the form:

$$\left(
                                                                    \begin{array}{cc}
                                                                      B & C \\
                                                                      0 & D \\
                                                                    \end{array}
                                                                  \right)$$

\noindent where $B$ is the adjacency matrix of  $Q/\pi$.\\

\noindent Let $\{v_1, \ldots,v_k\}$  be a set of  vectors of $\mathbb{C}^{m}$ such that  $B(v_i) = v_{i+1}$ for  $ 0 \leq i<k$ and  $B(v_{k}) = 0$. Consider next the vectors of $\mathbb{C}^{n}$ $\{w_1, \ldots, w_k\}$ whose coordinates are $w_i = (v_i, 0)$.  From this we deduce that there is a Jordan block of  $A$ associated with the eigenvalue $0$  with size bigger than  $k$, therefore using  the Theorem  \ref{forma de jordan 2}, we see that $\phi \dim\left(\frac{\mathbb{K} (Q/\pi)}{J^{2}}\right) \leq \phi \dim(A)$.
\end{proof}

\end{prop}

\begin{obs}
In the same conditions as the Proposition  \ref{cota equitativa} we can prove that each invariant subspace associates to a Jordan block of  $B$ is a subspace of a unique invariant subspace associated to a unique Jordan block of $A$.

\end{obs}

\begin{teo}\label{phi equitativa}

Let  $A = \frac{\mathbb{K} Q}{J^2}$ be an algebra without member, if $\pi$ is an equitable partition of  $Q$ with $m$ elements, we have that  $m-\phi \dim\left(\frac{\mathbb{K}(Q/\pi)}{J^2}\right) \leq n - \phi \dim(A)$.

\begin{proof} It follows from Theorem \ref{forma de jordan 2} that  $\phi \dim\left(\frac{\mathbb{K}(Q/\pi)}{J^2}\right) = k+1$ where $k$ is the size of the largest Jordan block associates with the eigenvalue  $0$ of the matrix $B$.
Since each of the others Jordan blocks of $B$, (when we discard one of size $k$) is of smaller or equal size to a Jordan block of $A$ discarding the block associates to the respective block of $B$ which we just eliminate, we have $0 \leq m-\phi \dim\left(\frac{\mathbb{K}(Q/\pi)}{J^2}\right) \leq n-\phi \dim (A)$.

\end{proof}
\end{teo}

\begin{coro}

Let $A = \frac{\mathbb{K} Q}{J^2}$ be an algebra without members, then  $\phi \dim(A)$ is  maximal if and only if
 $\phi \dim\left(\frac {\mathbb{K}(Q/\pi)}{J^{2}}\right)$ is maximal for each equitable partition $\pi$ of $Q$.\\

\begin{proof} We first prove that the condition is necessary.\\
Consider an equitable partition $\pi$, using the Theorem \ref{phi equitativa} we get that  $0 \leq m-\phi \dim\left(\frac{\mathbb{K}(Q/\pi)}{J^2}\right) \leq n - \phi \dim(A) = 0$, therefore  $\phi\left(\frac{\mathbb{K}(Q/\pi)}{J^2}\right) = 0$ and since $Q/\pi$ has $m$ vertices it follows that  $\fidim \left(\frac{ \mathbb{K}(Q/\pi)}{J^2}\right) = m$ is maximal.\\

\noindent The sufficient condition is clear, considering the trivial partition.
\end{proof}

\end{coro}

\noindent The following example shows how we can use the results above to see if an algebra does not have maximal $\phi$-dimension.

\begin{ej}
\emph{Let  $A = \frac{\mathbb{K} Q}{J^2}$ where  $Q$ has the following form:}

$$
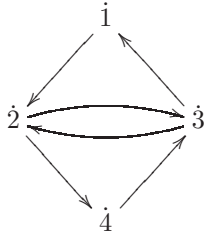
\xymatrix{ & \dot 1 \ar[dl] &  \\
    \dot 2 \ar@/^/[rr] \ar[dr] &  & \dot 3 \ar@/^/[ll] \ar[ul] \\
     & \dot 4 \ar[ur]&  } $$

\emph{Consider the equitable partition $\pi = \{\{1,4\}\{2,3\}\}$.  The algebra  $A' = \frac{\mathbb{K}(Q/\pi)}{J^2}$ has the following quiver:}

$$
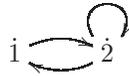
\xymatrix{ \dot 1 \ar@/^/[r] & \dot 2 \ar@/^/[l] \ar@(lu,ru) \\}$$\\

\emph{\noindent so its adjacency matrix has the following form:}

$$ \left(
  \begin{array}{cc}
    0 & 1 \\
    1 & 1 \\
  \end{array}
\right)$$

\emph{\noindent since this matrix is invertible we have that $\phi \dim(A') = 1$. Since the $\phi \dim (A')$ is not maximal, the  $\phi \dim (A)$ cannot be maximal.}

\end{ej}

\noindent The following example shows that we can have an algebra which does not have  maximal $\phi$-dimension; which has the property that every algebra obtained by a non trivial equitable partition  has maximal $\phi$-dimension.

\begin{ej}
\emph{The radical square zero algebra whose quiver is:}

$$ 
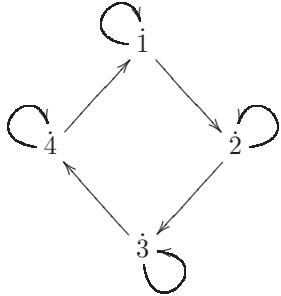
\xymatrix{ & \dot 1 \ar[rd] \ar@(l,u) & \\
               \dot 4 \ar[ur]  \ar@(l,u) &   & \dot 2 \ar[dl]  \ar@(r,u)\\
              &  \dot 3 \ar[ul]  \ar@(d,r)& }$$\\

\emph{\noindent does not have maximal $\phi$-dimension, nevertheless every square radical zero algebra obteined by an equitable partition has maximal  $\phi$-dimension.}
\end{ej}

\noindent It follows a short argument. It is not hard to show  that the algebra $A$ has $\phi$-dimension $2$.
It is clear that the possible equitable partitions are the trivial one and
\begin{itemize}
  \item $\pi_{1} = \{\{1,3\}, \{2,4\}\}$.
  \item $\pi_{2} = \{\{1,2,3,4\}\}$.
\end{itemize}

\begin{enumerate}
  \item The quiver of $Q/\pi_{1}$ has the form:
  $$
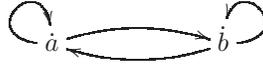
\xymatrix { & & \\   \dot a \ar@(l,u) \ar@/^/[rr] & &  \dot b \ar@/^/[ll] \ar@(r,u) \\ & & } $$
In this case the algebra has $\phi$-dimension $2$.

  \item The quiver of  $Q/\pi_{2}$  has the form:
  $$
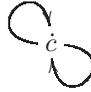
\xymatrix{ & & \\  & \dot c \ar@(l,u) \ar@(r,d) & \\ & & } $$
  In this case the algebra has $\phi$-dimension $1$.

\end{enumerate}

\subsection{Other  $\phi$-dimensions}

\begin{prop}
For all  $0 \leq k \leq n$ there are radical square algebras $A(k)$ with  $|Q_0| = n$ and $\fidim(A) = k$.

\begin{proof}
\begin{itemize} \item By Proposition \ref{autoinyectiva} we have that, $Q = Z_n$ if and only if $\fidim(A) = 0$.
  \item Taking  $A$  an algebra in the family on the Example \ref{phimaximal} we see that  $\fidim(A) = n$.
  \item Finally for  $0<k<n$,  we can take the radical square zero algebras given in the following quiver:

$$Q= 
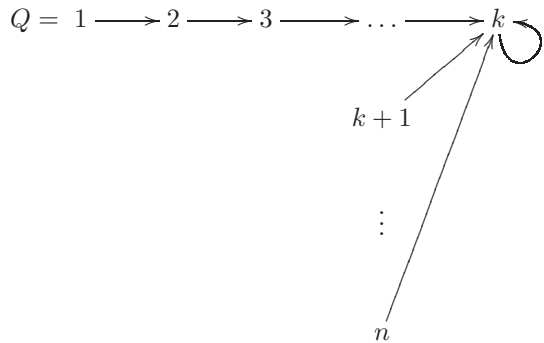
\xymatrix{  1 \ar[r]& 2 \ar[r] & 3 \ar[r]& \ldots \ar[r]& k \ar@(d,r)\\
&&&k+1 \ar[ur]&\\
&&& \vdots &\\
&&&n \ar[uuur]}$$
 It is clear  that  $\oplus_{\mathcal{S}_D} S = \oplus_{i =2}^k S_i$ therefore  $\phi(\oplus_{\mathcal{S}_D} S) = k-1$. Finally using the Proposition \ref{explicito}  we conclude that  $\fidim(A) = k$.
\end{itemize}

\end{proof}
\end{prop}

\begin{teo}\label{1}
If  $A$ has no member then  $\phi \dim(A)\leq 2$ if and only if the algebraic and the geometric multiplicity of the eigenvalue $0$ on the adjacency matrix coincide. We also have the following cases for algebras without member:

\begin{itemize}
  \item $\phi \dim(A) = 0$ if and only if its quiver is Z$_n$  if and only if the algebra is selfinjective (in this case $0$ is not an eigenvalue);
  \item $\phi \dim(A) = 1$ if and only if  the algebra is not selfinjective and $0$ is not an eigenvalue;
  \item $\phi \dim(A) = 2$ if and only if $0$ is an eigenvalue.
\end{itemize}

\begin{proof}
We can assume that $n\geq2$, otherwise the result is clear.\\

\noindent We know that $M=M_Q$ represents the linear transformation  $\overline{\Omega}|_{K_1}$ in the basis  $\{S\}_{S \in \mathcal{S}}$. If we consider the Jordan matrix $J$ of  $M$, we know that the rank of $J^k$ and $M^k$  are the same for each $k$.\\
\noindent Let us show that the condition is sufficient.\\
It is clear that if the algebraic and geometric multiplicities of the eigenvalue  $0$ coincide, then  $n-m_a(0)=rk(J) = rk(J^2) = \ldots =rk(J^k)$ for $k \in \mathbb{N}$. Therefore:

\begin{itemize}
  \item $\phi(A_0)= 0$ if $m_a(0)=0$ since in this case  $M$ is invertible,
  \item $\phi(A_0)= 1$ if $m_a(0)>0$ since in this case  $\overline{\Omega}|_{K_2}$ is invertible.
\end{itemize}

\noindent So, since $A$  has no member,  we have that $\phi \dim(A) = 1 + \phi(A_0)$, unless $A$ is selfinjective.\\

\noindent We now show that the condition is necessary.\\
If the algebraic multiplicity is bigger than the geometric we have that $rk (J) > rk(J^2)$ therefore $\phi(A_0)>1$ so we conclude that  $\phi \dim(A)>2$.
\end{proof}

\end{teo}

\noindent For the radical square zero algebras we know that the syzygy of any non projective module is a direct sum of simple modules. In particular if $Q$ has no member then any syzygy has infinite projective dimension, therefore every non projective module have infinite projective dimension, that is  $\fin (\frac{\mathbb{K} Q}{J^{2}}) = 0$.\\

\noindent Since the adjacency matrix of a $Z_n$ is symmetric if and only if $n\leq 2$, we get the following corollary:

\begin{coro}
Given a quiver $Q$, if  $M_Q$  is symmetric then  $\phi \dim(A) \leq 2$. Moreover if  $|Q_0|\geq 3$  then:
\begin{itemize}
  \item $\phi \dim(A) = 1$ if and only if $det(M) \neq 0$,
  \item $\phi \dim(A) = 2$ if and only if $det(M) = 0$.
\end{itemize}

\begin{proof} Since $M_Q$ is symmetric $M_Q$ is diagonalisable moreover  $Q$ has no member. Therefore $m_a(0) = m_g(0)$ using Theorem \ref{1} we get that  $\phi \dim(A) \leq 2$.\\

\noindent If $|Q_0|\geq 3$, then using Observation \ref{autoinyectiva} we see that  $A$ cannot be selfinjective, then  using Theorem \ref{1} we get:

\begin{itemize}
  \item $\phi \dim(A) = 1$ if $det(M) \neq 0$,
  \item $\phi \dim(A) = 2$ if $det(M) = 0$.
\end{itemize}

\end{proof}

\end{coro}

\subsection{The Igusa-Todorov function $\psi$}
In this section we characterise the radical square zero algebras with maximal $\psi$-dimension, and we will give a complete description of their quiver. It seemed, for us, natural to expect that the algebras with  maximal $\psi$-dimension would be the same as the ones with the maximal $\phi$-dimension. The next results show that this is not the case.
\begin{prop}

If $\fidim(A) = n$ then  $\psidim(A) = n$.

\begin{proof} The Theorem \ref{forma de jordan}, tell us that $\fidim(A) = n$ implies that  $M(Q)$ is empty. Since $Q$ has no sinks, using the Lemma \ref{findim0}, we have that $\fin (A)=0$. Therefore  $\fidim(A) = \psidim(A)$.
\end{proof}

\end{prop}

\begin{prop}\label{cota psi}

If $0< \fidim(A) \leq n-1$, then $\psidim(A) \leq n-1 + d$  where  $d$ is the length of the largest path in $M(Q)$ ending in a sink of  $Q$.

\begin{proof}
Given a semisimple non injective module $N$ with finite projective dimension we have that  $\pd(N) \leq d$. Therefore if $\phi(M) \neq 0$ then we have that  $\psi (M) \leq \phi(M) + d$.
\end{proof}

\end{prop}

\noindent We give next an example of an algebra with maximal $\psi$-dimension, accordingly with the bound given in the former proposition.

\begin{ej}\label{ej psi max} \emph{Let $A$ be the radical square zero algebra with the following quiver:}

$$  Q'= \ \ \ \xymatrix{  1 \ar[r]\ar@(u,l)& 2 \ar[r] & 3 \ar[r]& \ldots \ar[r]& n } $$

\emph{\noindent We know that there exists  $M$ such that $\phi(M) = n-1$ and  $\Omega(M) = \oplus_{i=1}^n(S_{i})$.
\noindent Since  $\Omega(\oplus_{i=1}^n(S_{i})) = \oplus_{i=1}^n(S_{i})$, we see that $\psi(M) = \phi(M)+\pd(S_2)= n-1+n-2 =2n-3$. Therefore $\psidim (A) = 2n-3$.}
\end{ej}

\begin{teo}
If $n>3$,  $\psidim (A)$ is maximal $(\psidim(A) = 2n-3)$ if and only if $M(Q)$ is totally ordered, $Q$ has a sink and $C(Q)$ has a unique vertex and various loops, and the quiver of  $Q$ has the form:

$$\xymatrix{C(Q)\ar@/^/[r] \ar@/_/[r]^{\vdots} & M(Q)}$$

\noindent where some of the arrows starting at  $C(Q)$ has to end at the source of $M(Q)$.

\begin{proof} In order to get the maximum of the  $\psi$-dimension, accordingly to Proposition \ref{cota psi},  we need to have, from what we saw
in the Example \ref{ej psi max}, that  $\fidim(A)= n-1$ and  $M(Q)$ is non empty. Since $\fidim(A)= n-1$ and $M(Q)\neq \emptyset$  accordingly to the Proposition  \ref{member ordenado},  we conclude that the member is completely  ordered.\\

\noindent If $Q$ has no sink, then the projective dimension of each simple module, and therefore for each non projective module, is not finite. From this we get, once again, that  $\psidim (A) = \fidim(A) = n-1$.\\

\noindent Finally if no arrows ends in the source of  $M(Q)$ then  $Q$ has a sink and a source, so $\fidim (A)<n-1$  and we get a contradiction.\\

\noindent The argument for the converse is similar to the one we gave in the last Example.

\end{proof}
\end{teo}

\section{Witnesses Modules}

We start this section with general Artin algebras. We begin with the description of the minimal number of non isomorphic summands which appear in a decomposition of an $A$-module $M$ such that $\phi(M) = \phi \dim (A)$.

\begin{defi}\label{testigos}
Let $A$ be an Artin algebra with $\phi \dim (A) = l < \infty$.  We say that an $A$-module $M$ is a \textbf{$\phi$-witness} if $\phi(M) = l$ and we say that it is a  \textbf{$\phi$-minimal witness} if $M$ is a $\phi$-witness and $rk \langle \add M\rangle$ is minimal.

\end{defi}

\begin{prop}
For any  Artin algebra $A$, we have that $\fin (A) = \phi \dim (A)$ is finite if and only if there is an indecomposable $A$-module $M$ which is a  $\phi$-witness.
\begin{proof}
Assume that $\fin (A) = \phi \dim (A)$ and is finite, so there is an indecomposable module $M$ whose projective dimension is equal to $\fin(A)$.  This indecomposable module is a $\phi$-witness for $A$.\\
Assume now that there is an indecomposable $A$-module $M$ which is a $\phi$-witness.\\
\noindent If $M$ has infinite projective dimension then $\phi \dim (A) = 0$, and since $\fin (A) \leq  \phi \dim (A)$ we have that $\phi \dim (A) = \fin (A)$. If $M$ has finite projective dimension then it is clear that this module realises the $\fin(A)$, and the result follows. \end{proof}
\end{prop}

\noindent We know, by Example \ref{ejemplito}, that the  $\phi$-dimension of an algebra and its finitistic dimension can be different. We also can use Theorem \ref{findim0} to get plenty of radical square zero algebras which are not selfinjective and have finitistic dimension equal to zero.

\begin{defi}\label{Gamma}
Given a basic $A$-module $M\cong \oplus_{i = 1}^{k} M_{i}$ with each $M_i$ indecomposable we define the \textbf{graph of M}, $\Gamma(M)$ in the following way:

\begin{itemize}
  \item $V(\Gamma(M)) = \{M_i\}_{i \in k}$.
  \item $\{M_i,M_j\}_{i\neq j} \in A(\Gamma(M))$ if  $\phi(M_i\oplus M_j) \geq 1$.
\end{itemize}

\end{defi}

\begin{prop}\label{comp_conexas}
Each connected component of the graph $\Gamma(M)$ of an $A$-module $M$ is complete.
\begin{proof}
Because $k=1, 2$ are trivial cases we assume that $M = \oplus_{i = 1}^{k} M_{i}$ is a decomposition of  $M$ with $k\geq 3$. Consider $M_{i_1}$, $M_{i_2}$ and $M_{i_3}$ in the same connected component of $\Gamma(M)$ such that $\{M_{i_1}, M_{i_2}\}$ and $\{M_{i_2}, M_{i_3}\}$ are edges of  $\Gamma(M)$. Since  $\{M_{i_1}, M_{i_2}\}$ is an edge of  $\Gamma(M)$ we have, by definition, that $\phi(M_{i_1}\oplus M_{i_2}) \geq 1$  so,
there exists a minimum $k_0\geq 1$ such that the rank of $\overline{\Omega}^{k}\langle [M_{i_1}],[ M_{i_2}]\rangle \leq 1$ for all $k \geq k_0$, and in this case we would get that $\overline{\Omega}^{k}[M_{i_1}] = \overline{\Omega}^{k}[M_{i_2}]$ for all $k \geq k_0$. In an analogous way we would have that $\overline{\Omega}^{k}[M_{i_2}] = \overline{\Omega}^{k}[M_{i_3}]$ for all $k \geq k_1 \geq 1$, putting tighter both statements we conclude that  $\overline{\Omega}^{k}[M_{i_1}] = \overline{\Omega}^{k}[M_{i_3}]$ for all $k \geq  k_{2}=$ max$\{k_0, k_1\}$. So we conclude that  $\phi(M_{i_1}\oplus M_{i_3}) \geq 1$ and finally we have that  $\{M_{i_1}, M_{i_3}\}$ is also an edge of  $\Gamma(M)$. \end{proof}

\end{prop}

\begin{prop}\label{phi-minimal}

Let  $M \in \mod A$ be such that $\phi(M) = l$. Assume that  $M = \oplus_{i = 1}^{k} M_i$ is a decomposition into indecomposable modules of $M$. If  $\Gamma(M)$ is connected then  there exists a direct summand of  $M$ of the form $M_{i_1}\oplus M_{i_2}$ such that $\phi (M_{i_1}\oplus M_{i_2}) = l$.
\begin{proof}
Assume that for all modules of the form $M_i \oplus M_j$ we have that  $\phi(M_i \oplus M_j)<l$. Let $l' =  $m\'ax$\{\phi(M_i \oplus M_j) $ with $ i\neq j \} < l$, using the former proposition we see that  $rk (\overline{\Omega}^{l'} \langle \add M \rangle) \leq 1$. Since $\phi (M) = l>l'$ we have that  $rk (\overline{\Omega}^{l} \langle \add M \rangle) = 0$. This implies that  $M$ is a module of finite projective dimension, in particular $dp(M) = \phi(M) = l$. Since we know that  $dp (M) = \maxi \{ dp(M_i) $ with $ i =1,\ldots, k \}$, we finally get that for some pair $(i,j)$ it is valid that $dp(M_i\oplus M_j) = \phi (M_i \oplus M_j) = l$ which is a contradiction. \end{proof}
\end{prop}

\subsection{Minimal $\phi$-witness modules for radical square zero algebras}

In this section we will analyse which kind of minimal $\phi$-witness modules we can get when $A$ is a radical square zero algebra and $\phi \dim(A) = n$ or if $A$ has a non empty member and $\phi \dim (A) = n-1$. That is we analyse which kind of minimal $\phi$-witness modules we can get when the algebra $A$ has maximal $\phi$-dimension.\\

\begin{prop} \label{conexo}
If $A$  is a radical square zero  algebra and $\phi \dim(A) = n$ or if $A$ has a non empty member and $\phi \dim (A) = n-1$, then $\Gamma(\oplus_{S \in S_{\mathcal{D}}} S)$ is connected.
\begin{proof} If we assume that  $\Gamma (\oplus_{S \in S_{\mathcal{D}}} S)$ is not connected then by Definition \ref{Gamma} and Proposition \ref{comp_conexas} we have that there exist different $i, j\in {1,2\ldots, n}$ such that  $\phi (S_i \oplus S_j) = 0$. This implies that  $\{[\Omega^{l}(S_i)], [\Omega^{l}(S_j)]\}$ is a linear independent set for all  $l \in \mathbb{N}$. It follows that $rk \overline{\Omega}^l(\langle \add( \oplus_{S \in S_{\mathcal{D}}} S ) \rangle) \geq 2$. Using Propositions \ref{invariante} and \ref{explicito} we see that the $\phi$-dimension of $A$ cannot be maximal, a contradiction. So $\Gamma(\oplus_{S \in S_{\mathcal{D}}} S)$ is connected. \end{proof}
\end{prop}

\begin{coro}
Let $A$ be a radical square zero algebra with maximal $\phi$-dimension. Then there is a module  $M = M_1 \oplus M_2$ with  $M_1$ and $M_2$ indecomposable such that  $\phi(M) = \phi \dim (A)$.
\begin{proof} Using the Proposition \ref{conexo} we get that  $\Gamma (\oplus_{S \in S_{\mathcal{D}}} S)$ is a connected graph and using the Proposition \ref{phi-minimal} we see that there exists a semisimple module $S_{i_1} \oplus S_{i_2}$ such that

\begin{itemize}
  \item $\phi (S_{i_1} \oplus S_{i_2}) = n-1$ if $\phi \dim(A) = n$ or

  \item $\phi (S_{i_1} \oplus S_{i_2}) = n-2$ if $\phi \dim(A) = n-1$ and $A$ has a non empty member.
\end{itemize}

\noindent Let $M_j$ with  $j = 1, 2$ where $\Omega (M_j) = S_{i_j}$, as in the proof of Lemma \ref{findim0}.  If we consider  $M = M_1\oplus M_2$ we have that $\phi(M) = \phi \dim (A)$. \end{proof}
\end{coro}

\noindent For every radical square zero algebra with finite  $\phi$-dimension there is a minimal $\phi$-witness module with at most two summands. The following give us an example of  a radical square zero algebra whose $\phi$-dimension is not maximal  and their  $\phi$-witness modules  are all indecomposable.

\begin{ej}

\emph{\noindent Consider the radical square zero algebra whose quiver is the following:}

$$Q = \xymatrix{ & 1 \ar@(ul,dl) \ar[r]& 2 \ar[r] & 3 \ar@/^/[r]& 4 \ar@(dr,ur) \ar@/^/[l]}$$

\emph{\noindent This algebra has an empty member, so $\{S_1, S_2, S_3, S_4\}$, the  classes  of simple modules, is a basis of $K_1$. \\
\noindent The values of   $\Omega$ in the simple modules are the following:}

\begin{itemize}
  \item $\Omega(S_1) = S_1 \oplus S_2$.
  \item $\Omega (S_2) = S_3$.
  \item $\Omega (S_3) = S_4$.
  \item $\Omega (S_4) = S_3 \oplus S_4$.
\end{itemize}

\emph{\noindent From this we see that  $rk(K_1) = 4 > 3 = rk (K_2) = rk(K_3)$, which implies that  $\phi \dim ( \frac{\Bbbk Q}{F^{2}}) = 2$. If we assume that  $M = M_1 \oplus M_2$ is the decomposition of a $\phi$-witness module $M$, we see that  $\{ [\Omega(M_1)], [\Omega(M_2)] \}$ must be linearly independent and  $\{[ {\Omega}^{2}(M_1)], [\Omega^{2}(M_2)] \}$ must be colinear. Since  $[S_2 ]+ [ S_3] - [S_4]$ is the  basis of the kernel of  $\overline{\Omega}$ we get, without lost of generality, that:}

\begin{itemize}
  \item  $S_2 \oplus S_3$ \emph{is a direct summand of } $\Omega(M_1)$.
  \item  $S_4$ \emph{ is direct summand of }$\Omega(M_2)$.
\end{itemize}

\emph{\noindent Since $P_1$ is the only projective indecomposable module, up to isomorphism, which has $S_2$ as a submodule and, on the other side,
$S_3$ is a submodule of the projective modules $P_2$ and $ P_4$, analysing the quiver $Q$, we can infer that that $M_1$ cannot be indecomposable.}
\end{ej}

\end{document}